\documentclass[a4paper, 11pt]{article}
\usepackage{amsmath,amsthm,wasysym,txfonts,indentfirst,fancyhdr,amscd,multirow,cite}
\usepackage[T1]{fontenc}
\usepackage{authblk}
\usepackage[utf8]{inputenc}
\usepackage{graphicx}

\usepackage{ifpdf}
\newtheorem{teo}{Theorem}[section]
\newtheorem{lemma}{Lemma}[section]
\newtheorem{defn}{Definition}[section]

\newtheorem{prop}{Proposition}[section]

\newtheorem{cor}{Corollary}[section]

\newcommand{\Z}{\mathbb{Z}}
\newcommand{\I}{\mathcal{I}}
\newcommand{\bfx}{\boldsymbol{x}}
\newcommand{\bfe}{\boldsymbol{e}}
\newcommand{\bfa}{\boldsymbol{a}}
\newcommand{\bfb}{\boldsymbol{b}}
\newcommand{\bfv}{\boldsymbol{v}}
\newcommand{\R}{\mathbb{R}}

\newcommand{\C}{\mathcal{C}}

\newcommand{\T}{\mathcal{T}}
\newcommand{\D}{\mathcal{D}}
\newcommand{\La}{\mathcal{L}}
\newcommand{\Po}{\mathcal{P}}
\newcommand{\G}{\mathcal{G}}

\newcommand{\h}{\mathcal{H}}
\newcommand{\K}{\mathcal{K}}
\newcommand{\cf}{\mathbf{c}}

\title{On the subgroup structure of the hyperoctahedral group in six dimensions}
\date{}
\author[1,3]{Emilio Zappa\thanks{ez537@york.ac.uk}}
\author[1,2,3]{Eric C. Dykeman}
\author[1,2,3]{Reidun Twarock}
\affil[1]{Department of Mathematics}
\affil[2]{Department of Biology}
\affil[3]{York Centre for Complex Systems Analysis,  University of York, York, UK}
           
\begin{document}
\maketitle                       

\begin{abstract}
	We investigate the subgroup structure of the hyperoctahedral group in six dimensions. In particular, we study the subgroups isomorphic to the icosahedral group. We classify the orthogonal crystallographic representations of the icosahedral group
	and analyse their intersections and subgroups, using results from graph theory and their spectra.
\end{abstract}

\section{Introduction}
The discovery of quasicrystals in 1984 by Shechtman et al. has spurred the mathematical and physical community to develop mathematical tools in order to study structures with non-crystallographic
symmetry.

Quasicrystals are alloys with five-fold or ten-fold symmetry in their atomic positions, and therefore they cannot be organised as (periodic) lattices. In crystallographic terms,
their symmetry group $G$ is non-crystallographic. However, the non-crystallographic symmetry leaves a lattice invariant in higher dimensions, providing an integral representation of $G$, 
usually referred to as \emph{crystallographic representation}. This representation is reducible and contains a two- or three- dimensional invariant subspace.  
This is the starting point to construct quasicrystals via the \emph{cut-and-project} method, described, among others, by Senechal \cite{senechal}, or as a model set \cite{moody}. 

In this paper we are interested in icosahedral symmetry. The icosahedral group $\I$ consists of all the rotations that leave a regular icosahedron invariant, and it is the largest of the finite subgroup of $SO(3)$. $\I$ contains elements of order $5$, therefore it is non-crystallographic in 3D; the (minimal) crystallographic 
representation of it is six-dimensional \cite{levitov}. The full icosahedral group, denoted by $\I_h$, contains also the reflections and is equal to $\I \times C_2$, where $C_2$ denotes the cyclic group of order 2. $\I_h$ is isomorphic to the Coxeter group $H_3$ and is made up of 120 elements. In this work we focus on the icosahedral group $I$, because it plays a central role in applications in virology \cite{giuliana}. However, our considerations apply equally to the larger group $\I_h$. 

Levitov and Rhyner \cite{levitov} classified the lattices in $\R^6$ that are left invariant by $\I$: 
there are, up to equivalence, exactly three lattices, usually referred to as \emph{icosahedral Bravais lattices}, namely the simple cubic (SC), body-centered cubic (BCC) and face-centered cubic (FCC). The point group of these lattices 
is the six-dimensional hyperoctahedral group, denoted by $B_6$, which is a subgroup of $O(6)$ and can be represented in the standard basis of $\R^6$ as the set of all $6 \times 6$ orthogonal and integral matrices. The subgroups
of $B_6$ which are isomorphic to the icosahedral group constitute the integral representations of it; among them, the crystallographic ones (following terminology from \cite{levitov}) are those which split, in $GL(6,\R)$, into two 
three-dimensional irreducible representations of $\I$. Therefore they carry two subspaces in $\R^3$ which are invariant under the action of $\I$ and can be used to model the quasiperiodic structures. 

The embedding of the icosahedral group into $B_6$ has been used extensively in the crystallographic literature. Katz \cite{katz}, Senechal \cite{senechal}, Kramer and Zeidler, \cite{zeidler}, Grimm \cite{grimm}, among others, 
start from a six-dimensional crystallographic representation of $\I$ to construct  three-dimensional Penrose tilings and icosahedral quasicrystals. Kramer \cite{kramer} and  
Indelicato et al. \cite{giuliana} also apply it to study structural transitions in quasicrystals. In particular, Kramer considers in $B_6$ a  representation of $\I$ and a representation of the octahedral group $\mathcal{O}$ which share a tetrahedral subgroup,  and defines a continous rotation (called Schur rotation) between cubic and icosahedral symmetry which preserves intermediate tetrahedral symmetry. Indelicato et al. define a transition between two icosahedral lattices as a 
continous path connecting the two lattice bases keeping some symmetry preserved, described by a maximal subgroup of the icosahedral group. The rationale behind this approach is that the two corresponding lattice groups share a common subgroup. These two approaches are shown to be 
related \cite{paolo}, hence the idea is that it is possible to study the transitions between icosahedral quasicrystals by considering two distinct crystallographic representations of $\I$ in $B_6$ which share a common subgroup. 

These papers motivate the idea of studying in some detail the subgroup structure of $B_6$. In particular, we focus on the subgroups isomorphic to the icosahedral group and its subgroups. Since the group is quite large (it has $2^66!$ elements), we use for computations the software \texttt{GAP} \cite{GAP},
 which is designed to compute properties of finite groups. More precisely, based on \cite{baake}, we generate the elements of $B_6$ in \texttt{GAP} as a subgroup of the symmetric group $S_{12}$ and then find the classes of subgroups isomorphic to the icosahedral group. Among them we isolate, using results from character theory, the class of crystallographic representations of $\I$. In order to study the subgroup structure of this class, we propose a method using graph theory and their spectra. In particular, we treat the class 
of crystallographic representations of $\I$ as a graph: we fix a subgroup $\G$ of $\I$ and say that two elements in the class are adjacent if their intersection is equal to a subgroup isomorphic to $\G$. We call the resulting graph $\G$-graph. These graphs are quite large and difficult to visualise; however,  by analysing their spectra \cite{doob} we can study in some detail their topology, hence describing the intersection and the subgroups shared by different representations. 
        
The paper is organised as follows. After recalling, in Section \ref{crystal_section}, the definitions of point group and lattice group, we define, in Section \ref{ico_section},
the crystallographic representations of the icosahedral group and the icosahedral lattices in 6 dimension. We provide, following \cite{haase}, a  method for the construction of the projection into 3D using tools from the representation theory of finite groups. 
In Section \ref{cube_section} we classify, with the help of \texttt{GAP}, the crystallographic representations of $\I$. In Section \ref{grafi_section} we study their subgroup structure, introducing the concept of $\G$-graph, where $\G$ is a subgroup of $\I$. 

\section{Lattices and non-crystallographic groups}\label{crystal_section}

Let $\bfb_i$, $i=1,\ldots n$ be a basis of $\R^n$, and let $B \in GL(n, \R)$ be the matrix whose columns are the components of $\bfb_i$ with respect to the canonical basis $\{\bfe_i, i=1,\ldots, n \}$ of $\R^n$. 
A \emph{lattice} in $\R^n$ is a $\Z$-free module of rank $n$ with basis $B$, i.e.
\begin{equation*}
\mathcal{L}(B) = \left\{ \bfx = \sum_{i=1}^n m_i \bfb_i : \; m_i \in \Z \right\}.
\end{equation*}
Any other lattice basis is given by $BM$, where $M \in GL(n,\Z)$, the set of invertible matrices with integral entries (whose determinant is equal to $\pm1$) \cite{artin}. 

The \emph{point group} of a lattice $\La$ is given by all the orthogonal transformations that leave the lattice invariant \cite{zanzotto}:
\begin{equation*}
\Po(B) = \{ Q \in O(n) : QB = BM, \; \exists M \in GL(n,\Z) \}.
\end{equation*}

We notice that, if $Q \in \Po(B)$, then $B^{-1}QB = M \in GL(n,\Z)$. In other words, the point group consists of all the orthogonal matrices which can be represented in the basis $B$ as 
integral matrices. The set of all these matrices constitute the \emph{lattice group} of the lattice:
\begin{equation*}
\Lambda(B) = \{ M \in GL(n,\Z): M = B^{-1}QB, \; \exists Q \in \Po(B) \}.
\end{equation*}
 
The lattice group provides an \emph{integral representation} of the point group, and these are related via the equation
\begin{equation*}
\Lambda(B) = B^{-1}\Po(B)B,
\end{equation*}
and moreover the following hold \cite{zanzotto}:
\begin{equation*}
\Po(BM) = \Po(B), \quad \Lambda(BM) = M^{-1}\Lambda(B)M, \quad M \in GL(n,\Z).
\end{equation*}

We notice that a change of basis in the lattice leaves the point group invariant, whereas the corresponding lattice groups are conjugated in $GL(n,\Z)$. 
Two lattices are \emph{inequivalent} if the corresponding lattice groups are not conjugated in $GL(n,\Z)$ \cite{zanzotto}.

As a consequence of the crystallographic restriction (see, for example, \cite{grimm}) five-fold symmetry is forbidden in dimensions 2 and 3, and therefore any group $G$ containing elements 
of order 5 cannot be the point group of a two- or three-dimensional lattice. We therefore call these groups \emph{non-crystallographic}. In particular, three-dimensional icosahedral lattices cannot exist. 
However, a non-crystallographic group leaves some lattices invariant in higher dimensions, and the smallest such dimension is called the minimal embedding dimension. 
Following \cite{levitov}, we introduce:
\begin{defn}\label{cryst} Let $G$ be a non-crystallographic group. A \emph{crystallographic representation} $\rho$ of $G$ is a $D$-dimensional representation of $G$ such that:
\begin{enumerate}
\item the charaters $\chi_{\rho}$ of $\rho$ are integers;
\item $\rho$ is reducible and contains a 2- or 3- dimensional representation of $G$.
\end{enumerate}
\end{defn}
We observe that the first condition implies that $G$ must be the subgroup of the point group of a $D$-dimensional lattice. The second condition tells us that $\rho$ contains either a two- or 
three-dimensional invariant subspace $E$ of $\R^D$, usually referred to as \emph{physical space} \cite{levitov}.
 
\section{$6D$ icosahedral lattices}\label{ico_section}

The icosahedral group $\I$ is generated by two elements, $g_2$ and $g_3$, such that $g_2^2 = g_3^3 = (g_2g_3)^5 = e$, where $e$ denotes the identity element. 
It has order 60 and it is isomorphic to $A_5$, the alternating group of order 5. Its character table is as follows (note that $\tau = \frac{\sqrt{5}+1}{2}$, and that $\C_2$, $\C_3$, $\C_5$ and $\C_5^2$ denote the conjugacy classes in $\I$ of 
$g_2$, $g_3$, $g_2g_3$ and $(g_2g_3)^2$, respectively, and the numbers denote the orders of each class):
\begin{center}
\begin{tabular}{l|c c c c c}
Irrep & $E$ & $12\C_5$ & $12\C_5^2$ & $15\C_2$ & $20\C_3$ \\
\hline
$A$ & 1 & 1 & 1 & 1 & 1 \\
$T_1$ & 3 & $\tau$ & 1-$\tau$ & -1 & 0 \\
$T_2$ & 3 & 1-$\tau$ & $\tau$ & -1 & 0 \\
$G$ & 4 & -1 & -1 & 0 & 1 \\
$H$ & 5 & 0 & 0 & 1 & -1 \\
\end{tabular}
\end{center}

From the character table we see that the (minimal) crystallographic representation of $\I$ is 6-dimensional and is given by $T_1 \oplus T_2$. 
Therefore, $\I$ leaves a lattice in $\R^6$ invariant. \cite{levitov} proved that the three inequivalent lattices of this type, mentioned in the introduction and referred to as \emph{icosahedral (Bravais) lattices}, are given by, respectively: 
\begin{equation*}
\La_{SC} = \left\{ \bfx =(x_1, \ldots, x_6) : x_i \in \Z \right\},
\end{equation*}
 \begin{equation*}
\La_{BCC} = \left\{ \bfx = \frac{1}{2}(x_1, \ldots, x_6) : x_i \in \Z, \; x_i = x_j \; \text{mod}2, \forall i,j=1,\ldots,6 \right\},
\end{equation*}
\begin{equation*}
\La_{FCC} = \left\{ \bfx = \frac{1}{2}(x_1, \ldots, x_6) : x_i \in \Z, \; \sum_{i=1}^6 x_i = 0 \; \text{mod}2 \right\}.
\end{equation*}

We note that a basis of the SC lattice is the canonical basis of $\R^6$. Its point group is given by
\begin{equation}\label{B_6}
\Po_{SC} = \{ Q \in O(6) : Q = M \in GL(6,\Z) \} = O(6) \cap GL(6,\Z) \simeq O(6,\Z),
\end{equation}
which is the \emph{hyperoctahedral group} in dimension $6$, denoted by $B_6$ \cite{baake}. All the three lattices have point group $B_6$, whereas their lattice groups are different and, 
indeed, they are not conjugate in $GL(6,\Z)$ \cite{levitov}. 

Let $\h$ be a subgroup of $B_6$ isomorphic to $\I$. $\h$ provides a (faithful) integral and orthogonal representation of $\I$. Moreover, if $\h \simeq T_1 \oplus T_2$ in $GL(6,\R)$, then $\h$ is also
crystallographic (in the sense of Definition \ref{cryst}). All the other crystallographic representations are given by $B^{-1}\h B$, where 
$B \in GL(6,\R)$ is a basis of an icosahedral lattice in $\R^6$. Therefore we can focus our attention, without loss of generality, on the orthogonal crystallographic representations.

\subsection{Projection operators}\label{proiezione}

Let $\h$ be a crystallographic representation of the icosahedral group. $\h$ splits into two 3-dimensional irreducible representations (IRs), $T_1$ and $T_2$, in $GL(6,\R)$. 
This means that there exists a matrix $R \in GL(6,\R)$ such that
\begin{equation}\label{R}
\h':= R^{-1}\h R = \left( \begin{array}{cc}
T_1 & 0 \\
0 & T_2 
\end{array} \right).
\end{equation}

The two IRs $T_1$ and $T_2$ leave two three-dimensional subspaces invariant, which are usually referred to as the \emph{physical} (or parallel) space $E^{\parallel}$ and the 
\emph{orthogonal} space $E^{\perp}$ \cite{katz}. In order to find the matrix $R$ (which is not unique in general), we follow \cite{haase} and use results from the representation theory of finite groups  (for proofs and further results see,
example, \cite{fulton}). In particular, let $\Gamma : G \rightarrow GL(n,F)$ be an $n$-dimensional representation of a finite group $G$ over a field $F$ ($F = \R, \mathbb{C}$). 
By Maschke's theorem, $\Gamma$ splits, in $GL(n,F)$, as $m_1 \Gamma_1 \oplus \ldots  \oplus m_r \Gamma_r$, where $\Gamma_i : G \rightarrow GL(n_i, F)$ is an $n_i$-dimensional IR of $G$.
Then the \emph{projection operator} 
$P_i : F^n \rightarrow F^{n_i}$ is given by
\begin{equation}\label{proj}
P_i:=\frac{n_i}{|G|} \sum_{g \in G} \chi^*_{\Gamma_i}(g) \Gamma(g),
\end{equation}
where $\chi_{\Gamma_i}^*$ denotes the complex conjugate of the character of the representation $\Gamma_i$. This operator is such that its image $\text{Im}(P_i)$ is equal to a $n_i$-dimensional subspace $V_i$ of $F^n$ invariant under $\Gamma_i$.
In our case, we have two projection operators, $P_i : \R^6 \rightarrow \R^3$, $i = 1,2$, corresponding to the IRs $T_1$ and $T_2$, respectively. We assume the image of $P_1$, $\text{Im}(P_1)$,
to be equal to $E^{\parallel}$, and $\text{Im}(P_2) = E^{\perp}$. If $\{\bfe_j, j=1,\ldots,6 \}$ is the canonical basis of $\R^6$, 
then a basis of $E^{\parallel}$ (respectively $E^{\perp}$) can be found considering the set $\{\hat{\bfe}_j:=P_i\bfe_j, j=1,\ldots, 6\}$ for $i=1$ (respectively $i=2$) and then extracting a basis $\mathcal{B}_i$ from it. Since $\text{dim}E^{\parallel} = \text{dim}E^{\perp}=3$, 
we obtain $\mathcal{B}_i = \{\hat{\bfe}_{i,1},\hat{\bfe}_{i,2},\hat{\bfe}_{i,3} \}$, for $i=1,2$. The matrix $R$ can be thus written as
\begin{equation}\label{matrice_proiezione}
R = \left(\underbrace{\hat{\bfe}_{1,1},\hat{\bfe}_{1,2},\hat{\bfe}_{1,3}}_{\text{basis of} \; E^{\parallel}},\underbrace{\hat{\bfe}_{2,1},\hat{\bfe}_{2,2},\hat{\bfe}_{2,3}}_{\text{basis of} \; E^{\perp}}\right).
\end{equation}

Denoting by $\pi^{\parallel}$ and $\pi^{\perp}$ the $3 \times 6$ matrices which represent $P_1$ and $P_2$ in the bases $\mathcal{B}_1$ and $\mathcal{B}_2$, respectively, we have, by linear algebra
\begin{equation}\label{inv}
R^{-1} = \left( \begin{array}{c}
\pi^{\parallel} \\
\pi^{\perp} 
\end{array} \right).
\end{equation}

Since $R^{-1}\h = \h' R^{-1}$ (cf. \eqref{R}), we obtain
\begin{equation}\label{comm}
\pi^{\parallel}(\h(g)\bfv) = T_1(g)(\pi^{\parallel}(\bfv)), \quad \pi^{\perp}(\h(g)\bfv) = T_2(g)(\pi^{\perp}(\bfv)),
\end{equation}
for all $g \in \I$ and $\bfv \in \R^6$. In particular, the following diagramme commutes for all $g \in G$:
\begin{equation}\label{diagramma}
\begin{CD}
\R^6 @>\h(g)>> \R^6\\
@VV\pi^{\parallel}V @VV\pi^{\parallel}V\\
E^{\parallel} @>T_1(g)>> E^{\parallel}
\end{CD}
\end{equation}

The set $(\h, \pi^{\parallel})$ is the starting point for the construction of quasicrystals via the cut-and-project method (\cite{senechal}, \cite{paolo}). 

\section{Crystallographic representations of $\I$}\label{cube_section}
From the previous section it follows that the six-dimensional hyperoctahedral group $B_6$ contains all the (minimal) orthogonal crystallographic representations 
of the icosahedral group. In this section we classify them, with the help of the computer software programme \texttt{GAP} \cite{GAP}. 

\subsection{Representations of the hyperoctahedral group $B_6$}
Permutation representations of the $n$-dimensional hyperoctahedral group $B_n$ in terms of elements of $S_{2n}$, the symmetric group of order $(2n)!$, have been described in \cite{baake}. In this subsection we review these results,
since they allow us to generate $B_6$ in \texttt{GAP} and further study its subgroup structure.

It follows from \eqref{B_6} that $B_6$ consists of all the orthogonal integral matrices. A matrix $A=(a_{ij})$ of this kind must satisfy $AA^T = I_6$, the identity matrix of order 6, 
and have integral entries only. It is easy to see that these conditions imply that $A$ has entries in $\{0,\pm 1\}$ and  each row and column contains 1 or $-1$ only once. 
These matrices are called \emph{signed permutation matrices}. It is straightforward to see that any $A \in B_6$ can be written in the form $NQ$, where $Q$ is a $6 \times 6$ 
permutation matrix and $N$ is a diagonal matrix with each diagonal entry being either 1 or -1. We can thus associate with each matrix in $B_6$ a pair $(\bfa, \pi)$, where $\bfa \in \Z_2^6$ is 
a vector given by the diagonal elements of $N$, and $\pi \in S_6$ is the permutation associated with $Q$. The set of all these pairs constitute a group (called the \emph{wreath product} of $\Z_2$ and $S_6$, and denoted by 
$\Z_2 \wr S_6$, \cite{humpreys}) with the multiplication rule given by
\begin{equation*}
(\bfa, \pi)(\bfb, \sigma):=(\bfa_{\sigma}+_2 \bfb, \pi\sigma),
\end{equation*}
where $+_2$ denotes addition modulo 2 and 
\begin{equation*}
(\bfa_{\sigma})_k:=a_{\sigma(k)}, \quad \bfa = (a_1, \ldots, a_6).
\end{equation*}
$\Z_2 \wr S_6$ and $B_6$ are isomorphic, an isomorphism $T$ being the following:
\begin{equation}\label{isomorfismo}
[T(\bfa,\pi)]_{ij}:=(-1)^{a_j}\delta_{i,\pi(j)}.
\end{equation}

It immediately follows that $|B_6| = 2^66! = 46,080$. A set of generators is given by
\begin{equation}\label{generatori}
\alpha:=(\mathbf{0}, (1,2)), \quad \beta:=(\mathbf{0},(1,2,3,4,5,6)), \quad \gamma:=((0,0,0,0,0,1),\text{id}_{S_6}),
\end{equation}
which satisfy the relations 
\begin{equation*}
\alpha^2 = \gamma^2 = \beta^6 = (\mathbf{0},\text{id}_{S_6}).
\end{equation*}

Finally, the function $\phi : \Z_2 \wr S_6 \rightarrow S_{12}$ defined by
\begin{equation}\label{morfismo}
\phi(\bfa, \pi)(k):= \left\{
\begin{aligned}
&\pi(k)+6a_k \quad \text{if} \; 1\leq k \leq 6 \\
&\pi(k-6)+6(1-a_{k-6}) \quad \text{if} \; 7 \leq k \leq 12
\end{aligned}
\right.
\end{equation} 
is injective and maps any element of $\Z_2 \wr S_6$ into a permutation of $S_{12}$, and provides a faithful permutation representation of $B_6$ as a subgroup of $S_{12}$. 
Combining \eqref{isomorfismo} with the inverse of \eqref{morfismo} we get the function
\begin{equation}\label{iso}
\psi:= T \circ \phi^{-1} : S_{12} \rightarrow B_6
\end{equation}
which can be used to map a permutation into an element of $B_6$.

\subsection{Classification}

\begin{figure}
\includegraphics[scale=0.43]{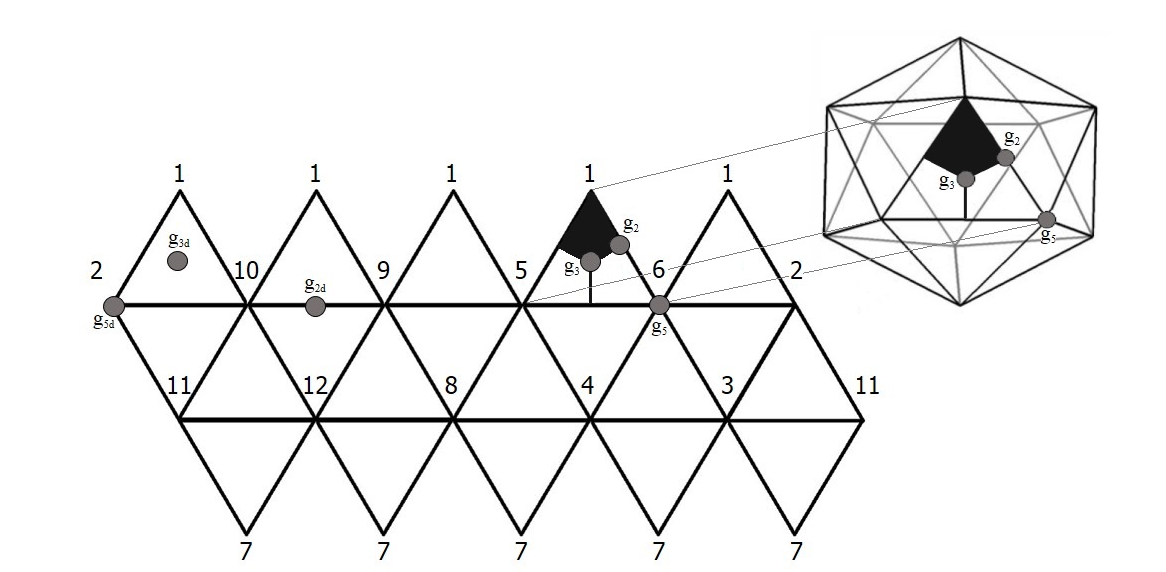}
\caption{A planar representation of an icosahedral surface, showing our labelling convention for the vertices; the dots represent the locations of the symmetry axes corresponding to the generators of the icosahedral group and its subgroups. The kite highlighted is a fundamental domain of the icosahedral group.}
\label{ico}
\end{figure}

In this subsection we classify the orthogonal crystallographic representations of the icosahedral group.  We start by recalling a standard way to construct such a representation, following \cite{zappa}. 
We consider a regular icosahedron and we label each vertex by a number from 1 to 12, so that the vertex opposite to vertex $i$ is labelled by $i+6$ (see Figure \ref{ico}). 
This labelling induces a permutation representation $\sigma : \I \rightarrow S_{12}$ given by
\begin{align*}
\sigma(g_2) &= (1,6)(2,5)(3,9)(4,10)(7,12)(8,11), \\
\sigma(g_3) & = (1,5,6)(2,9,4)(7,11,12)(3,10,8).
\end{align*}
Using \eqref{iso} we obtain a representation $\hat{\I} : \I \rightarrow B_6$ given by 
\begin{equation}\label{gen}
\hat{\I}(g_2) =  \left( \begin{array}{cccccc}
0 & 0 & 0 & 0 & 0 & 1 \\
0 & 0 & 0 & 0 & 1 & 0 \\
0 & 0 & -1 & 0 & 0 & 0 \\
0 & 0 & 0 & -1 & 0 & 0 \\
0 & 1 & 0 & 0 & 0 & 0 \\
1 & 0 & 0 & 0 & 0 & 0 
\end{array} \right), \quad \hat{\I}(g_3) = \left( \begin{array}{cccccc}
0 & 0 & 0 & 0 & 0 & 1 \\
0 & 0 & 0 & 1 & 0 & 0 \\
0 & -1 & 0 & 0 & 0 & 0\\
0 & 0 & -1 & 0 & 0 & 0 \\
1 & 0 & 0 & 0 & 0 & 0 \\
0 & 0 & 0 & 0 & 1 & 0 
\end{array} \right).
\end{equation}

We see that $\chi_{\hat{\I}}(g_2) = -2$ and $\chi_{\hat{\I}}(g_3) = 0$, so that, by looking at the character table of $\I$, we have
\begin{equation*}
 \chi_{\hat{\I}} = \chi_{T_1} + \chi_{T_2},
\end{equation*}
which implies, using Maschke's theorem \cite{fulton}, that $\hat{\I} \simeq T_1 \oplus T_2$ in $GL(6,\R)$. 
Therefore, the subgroup $\hat{\I}$ of $B_6$ is a crystallographic representation of $\I$. 

Before we continue, we recall the following \cite{humpreys}: 
\begin{defn}\label{conjugacy_class} Let $H$ be a subgroup of a group $G$. The \emph{conjugacy class of $H$ in $G$} is the set 
\begin{equation*}
\C_G(H) := \{ gHg^{-1} : g \in G \}.
\end{equation*}
\end{defn}

In order to find all the other crystallographic representations, we use the following scheme:
\begin{enumerate}
\item We generate $B_6$ as a subgroup of $S_{12}$ using \eqref{generatori} and \eqref{morfismo};
\item we list all the conjugacy classes of the subgroups of $B_6$ and find a representative for each class;
\item we isolate the classes whose representatives have order $60$;
\item we check if these representatives are isomorphic to $\I$;
\item we map these subgroups of $S_{12}$ into $B_6$ using \eqref{iso} and isolate the crystallographic ones by checking the characters; denoting by $S$ the representative, we decompose
$\chi_S$ as
\begin{equation*}
 \chi_S = m_1 \chi_{A} + m_2 \chi_{T_1} + m_3 \chi_{T_2} + m_4 \chi_{G} + m_5 \chi_{H}, \quad m_i \in \mathbb{N}, \; i=1,\ldots,5.
\end{equation*}
Note that $S$ is crystallographic if and only if $m_2 = m_3 =1$ and $m_1=m_4=m_5 = 0$.
\end{enumerate}

We implemented steps 1-4 in \texttt{GAP} (see Appendix). There are three conjugacy classes of subgroups isomorphic to $\I$ in $B_6$. 
Denoting by $S_i = < g_{2,i},g_{3,i} >$ the representatives of the classes returned by \texttt{GAP}, we have, using \eqref{iso},
\begin{equation*}
 \chi_{S_1}(g_{2,1}) = 2, \; \chi_{S_1}(g_{3,1}) = 3 \Rightarrow \chi_{S_1} = 2\chi_A + \chi_G \Rightarrow S_1 \simeq 2A \oplus G,  
\end{equation*}
\begin{equation*}
 \chi_{S_2}(g_{2,2}) = -2, \; \chi_{S_2}(g_{3,2}) = 0 \Rightarrow \chi_{S_2} = \chi_{T_1} + \chi_{T_2} \Rightarrow S_2 \simeq T_1 \oplus T_2,  
\end{equation*}
\begin{equation*}
 \chi_{S_3}(g_{2,3}) = 2, \; \chi_{S_3}(g_{3,3}) = 0 \Rightarrow \chi_{S_3} = \chi_A + \chi_H \Rightarrow S_3 \simeq A \oplus H. 
\end{equation*}

Since $2A$ is decomposable into 2 one-dimensional representations, it is not strictly speaking 2D in the sense of Definition \ref{cryst}, and  
as a consequence, only the second class contains the crystallographic representations of $\I$. A computation in \texttt{GAP} shows that its order is 192. We thus have the following
\begin{prop}\label{class}
The crystallographic representations of $\I$ in $B_6$ form a unique conjugacy class in the set of all the classes of subgroups of $B_6$, and its order is equal to 192. 
\end{prop}

We briefly point out that the other two classes of subgroups isomorphic to $\I$ in $B_6$ have an interesting algebraic intepretation. First of all, we observe that $B_6$ is an \emph{extension}
of $S_6$, since according to \cite{humpreys},
\begin{equation*}
B_6 / \Z_2^6 \simeq (\Z_2 \wr S_6) / \Z_2^6 \simeq S_6.
\end{equation*}

Following \cite{rotman}, it is possible to embed the symmetric group $S_5$ into $S_6$ in two different ways. The canonical embedding is achieved by fixing a point in $\{1, \ldots, 6 \}$ and permuting 
the other five, whereas the other embedding is by means of the so-called "exotic map" $\varphi : S_5 \rightarrow S_6$, which acts on the six 5-Sylow subgroups of $S_5$ by conjugation. Recalling that the icosahedral
group is isomorphic to the alternating group $A_5$, which is a normal subgroup of $S_5$, then the canonical embedding corresponds to the representation $2A\oplus G$ in $B_6$, while the exotic one corresponds to the 
representation $A \oplus H$. 

In what follows, we will consider the subgroup $\hat{\I}$ previously defined as a representative of the class of the crystallographic representations of $\I$, and denote this class by $\C_{B_6}(\hat{\I})$. 

Recalling that two representations $D^{(1)}$ and $D^{(2)}$ of a group $G$ are said to be \emph{equivalent} if there are related via a similarity transformation, i.e. there exists an invertible
matrix $S$ such that
\begin{equation*}
D^{(1)} = S D^{(2)} S^{-1},
\end{equation*}
then an immediate consequence of Proposition \ref{class} is the following
\begin{cor}
Let $\h_1$ and $\h_2$ be two orthogonal crystallographic representations of $\I$. Then $\h_1$ and $\h_2$ are equivalent in $B_6$.
\end{cor}

We observe that the determinant of the generators of $\hat{\I}$ in \eqref{gen} is equal to 1, so that $\hat{\I} \in B_6^+:= \{ A \in B_6 : \text{det}A = 1 \}$. Proposition \ref{class} implies that all the
crystallographic representations belong to $B_6^+$. The remarkable fact is that they split into \emph{two} different classes in $B_6^+$. To see this, we first need to generate $B_6^+$. In particular,
with \texttt{GAP} we isolate the subgroups of index $2$ in $B_6$, which are normal in $B_6$, and then, using \eqref{iso}, we find the one  whose generators have determinant equal to 1. In particular, we have
\begin{align*}
B_6^+ = & <(1,2,6,4,3)(7,8,12,10,9),(5,11)(6,12), \\
&  (1,2,6,5,3)(7,8,12,11,9),(5,12,11,6)>.
\end{align*}

We can then apply the same procedure to find the crystallographic representations of $\I$, and see that they split into two classes, each one of size 96. 
Again we can choose $\hat{\I}$ as a representative for one of these classes; a representative $\hat{\K}$ for the other one is given by
\begin{equation}\label{class_2}
\hat{\K} = \left< \left( \begin{array}{cccccc}
0 & 1 & 0 & 0 & 0 & 0 \\
1 & 0 & 0 & 0 & 0 & 0 \\
0 & 0 & -1 & 0 & 0 & 0 \\
0 & 0 & 0 & 0 & 0 & 1 \\
0 & 0 & 0 & 0 & -1 & 0 \\
0 & 0 & 0 & 1 & 0 & 0 
\end{array} \right), \left( \begin{array}{cccccc}
0 & 0 & 0 & 1 & 0 & 0 \\
1 & 0 & 0 & 0 & 0 & 0 \\
0 & 0 & 0 & 0 & 0 & -1 \\
0 & 1 & 0 & 0 & 0 & 0 \\
0 & 0 & -1 & 0 & 0 & 0 \\
0 & 0 & 0 & 0 & 1 & 0 
\end{array} \right) \right>.
\end{equation} 

We note that in the more general case of $\I_h$, we can construct the crystallographic representations of $\I_h$ starting from the crystallographic representations of $\I$. First of all, we recall that $\I_h = I \times C_2$, where $C_2$ is the cyclic group of order 2.  Let $\h$ be a crystallographic representation of $\I$ in $B_6$, and let $\Gamma = \{1,-1 \}$ be a one-dimensional representation of $C_2$. 
Then the representation $\hat{\h}$ given by
\begin{equation*}
\hat{\h} := \h \otimes \Gamma,
\end{equation*}
where $\otimes$ denotes the tensor product of matrices, is a representation of $\I_h$ in $B_6$ and it is crystallographic in the sense of Definition \ref{cryst} \cite{fulton}.

\subsection{Projection into the 3D space}

\begin{table}[!t]
\caption{Explicit forms of the IRs $T_1$ and $T_2$ with $\hat{\I} \simeq T_1 \oplus T_2$.}
\begin{center}
\begin{tabular}{c c c}
Generator & Irrep $\Gamma_1$ & Irrep $\Gamma_2$ \\
\hline
$g_2$ & $\frac{1}{2} \left( \begin{array}{ccc}
\tau-1 & 1 & \tau \\
1 & -\tau & \tau-1 \\
\tau & \tau-1 & -1
\end{array} \right)$ &  $\frac{1}{2} \left( \begin{array}{ccc}
\tau-1 & -\tau & -1 \\
-\tau & -1 & \tau-1 \\
-1 & \tau-1 & -\tau 
\end{array} \right)$ \\
$g_3$  & $\frac{1}{2} \left( \begin{array}{ccc}
\tau & \tau-1 & 1 \\
1-\tau & -1 & \tau \\
1 & -\tau & 1-\tau
\end{array} \right)$ &  $\frac{1}{2} \left( \begin{array}{ccc}
-1 & 1-\tau & -\tau \\
\tau-1 & \tau & -1 \\
\tau & -1 & 1-\tau 
\end{array} \right)$
\end{tabular}
\end{center}
\label{irreps}
\end{table}

We study in detail the projection into the physical space $E^{\parallel}$ using the methods described in Section \ref{proiezione}.

Let $\hat{\I}$ be the crystallographic representation of $\I$ given in \eqref{gen}. Using \eqref{proj} with $n_i = 3$ and $|G| = |\I| = 60$ we obtain the following projection operators
\begin{equation*}
P_1 = \frac{1}{2\sqrt{5}}\left( \begin{array}{cccccc}
\sqrt{5} & 1 & -1 & -1& 1 & 1 \\
1 & \sqrt{5} &1 &-1 &-1 &1 \\
-1& 1 & \sqrt{5} & 1& -1& 1 \\
-1 &-1 &1 &\sqrt{5} & 1 &1 \\
1& -1 & -1 & 1 & \sqrt{5} & 1 \\
1 & 1 & 1 & 1 & 1 & \sqrt{5}
\end{array} \right),
\end{equation*}
\begin{equation*}
P_2 =  \frac{1}{2\sqrt{5}}\left( \begin{array}{cccccc}
\sqrt{5} & -1 & 1 & 1& -1 & -1 \\
-1 & \sqrt{5} &-1 & 1 & 1 & -1 \\
1& -1 & \sqrt{5} & -1& 1& -1 \\
1 & 1 &-1 &\sqrt{5} & -1 &-1 \\
-1 & 1 & 1 & -1 & \sqrt{5} & -1 \\
-1 & -1 & -1 & -1 & -1 & \sqrt{5}
\end{array} \right) .
\end{equation*}

The rank of these operators is equal to 3. We choose as a basis of $E^{\parallel}$ and $E^{\perp}$ the following linear combination of the columns $\cf_{i,j}$ of the projection operators $P_i$, for $i = 1,2$ and $j =1, \ldots, 6$:
\begin{equation*}
\left( 
\underbrace{\frac{\cf_{1,1}+\cf_{1,5}}{2}, \frac{\cf_{1,2}-\cf_{1,4}}{2}, \frac{\cf_{1,3}+\cf_{1,6}}{2}}_{\text{basis of $E^{\parallel}$}}, \underbrace{\frac{\cf_{2,1}-\cf_{2,5}}{2}, \frac{\cf_{2,2}+\cf_{2,4}}{2},\frac{\cf_{2,3}-\cf_{2,6}}{2}}_ {\text{basis of $E^{\perp}$}}
\right).
\end{equation*} 

With a suitable rescaling, we obtain the matrix $R$ given by
\begin{equation*}
R =  \frac{1}{\sqrt{2(2+\tau)}} \left( \begin{array}{cccccc}
\tau & 1 & 0 & \tau & 0 & 1 \\
0 & \tau & 1 & -1 & \tau & 0 \\
-1 & 0 & \tau & 0 & -1& \tau \\
0 & -\tau & 1 & 1 & \tau & 0 \\
\tau & -1 & 0 & -\tau & 0 & 1 \\
1 & 0 & \tau & 0 & -1 & -\tau
\end{array} \right) .
\end{equation*}

The matrix $R$ is orthogonal and reduces $\hat{\I}$ as in \eqref{R}. In Table \ref{irreps} we give the explicit forms of the reduced representation. 
The matrix representation in $E^{\parallel}$ of $P_1$ is given by (see \eqref{inv})

\begin{equation*}
\pi^{\parallel} =  \frac{1}{\sqrt{2(2+\tau)}} \left( \begin{array}{cccccc}
\tau & 0 & -1 & 0 & \tau & 1 \\
1 & \tau & 0 & -\tau & -1 & 0 \\
0 & 1 & \tau & 1 & 0 & \tau 
\end{array} \right).
\end{equation*}

The orbit $\{T_1(\pi^{\parallel}(\bfe_j)) \}$, where $\{\bfe_j, j = 1,\ldots, 6 \}$ is the canonical basis of $\R^6$, represents  a regular icosahedron in 3D centered at the origin (\cite{senechal}, \cite{katz} and \cite{giuliana}).
  
Let $\K$ be another crystallographic representation of $\I$ in $B_6$. By Proposition \ref{class},  $\K$ and $\hat{\I}$ are conjugated in $B_6$. Consider $M \in B_6$ such that $M \hat{\I} M^{-1} = \K$ and let $ S = MR$. We have
\begin{equation*}
S^{-1} \K S = (MR)^{-1} \K (MR) = R^{-1} M^{-1} \K M R = R^{-1} \hat{\I} R = T_1 \oplus T_2. 
 \end{equation*}

Therefore it is possible, with a suitable choice of the reducing matrices, to reduce all the crystallographic representations of $\I$ in $B_6$ into the same irreps.

\section{Subgroup structure}\label{grafi_section}

\begin{table}[!t]
\caption{Non-trivial subgroups of the icosahedral group: $\T$ stands for the tetrahedral group, 
$\D_{2n}$ for the dihedral group of order $2n$, and $C_n$ for the cyclic group of order $n$. }
\begin{center}
\begin{tabular}{c c c c}
Subgroup & Generators & Relations & Order \\
\hline
$\mathcal{T}$ & $g_2, g_{3d}$ & $g_2^2=g_{3d}^3=(g_2g_{3d})^3=e$ & 12 \\
$\mathcal{D}_{10}$ & $g_{2d},g_{5d}$ & $g_{2d}^2=g_{5d}^5=(g_{5d}g_{2d})^2=e$ & 10   \\
$\mathcal{D}_{6}$ & $g_{2d},g_3$ & $g_{2d}^2=g_3^3=(g_3g_{2d})^2=e$ & 6 \\
$C_5$ & $g_{5d}$ & $g_{5d}^5=e$ & 5  \\
$\mathcal{D}_4$ & $g_{2d},g_2$ & $g_{2d}^2=g_2^2=(g_2g_{2d})^2=e$ & 4  \\
$C_3$ & $g_3$ & $g_3^3=e$ & 3 \\
$C_2$ & $g_2$ & $g_2^2=e$ & 2 \\
\end{tabular}
\end{center}
\label{ico_sgp}
\end{table}

\begin{table}[!t]
\caption{Permutation representations of the generators of the subgroups of the icosahedral group.}
\begin{center}
\begin{tabular}{l}
\hline
$\sigma(g_2) = (1,6)(2,5)(3,9)(4,10)(7,12)(8,11)$ \\
$\sigma(g_{2d})  = (1,12)(2,8)(3,4)(5,11)(6,7)(9,10)$ \\
$\sigma(g_3)  = (1,5,6)(2,9,4)(7,11,12)(3,10,8)$ \\
$\sigma(g_{3d})  = (1,10,2)(3,5,12)(4,8,7)(6,9,11)$ \\
$\sigma(g_5)  = (1,2,3,4,5)(7,8,9,10,11)$ \\
$\sigma(g_{5d})  = (1,10,11,3,6)(4,5,9,12,7)$ \\
\hline
\end{tabular}
\end{center}
\label{perm_rep}
\end{table}

\begin{table}[!t]
\caption{Order of the classes of subgroups of the icosahedral group in $\I$ and $B_6$.}
\begin{center}
\begin{tabular}{c c c}
Subgroup & $|\C_{\I}(\G)|$ & $|\C_{B_6}(\K_{\G})|$ \\
\hline
$\T$ & 5 & 480\\ 
$\D_{10}$ & 6 & 576\\
$\D_6$ & 10 & 960\\
$\D_4$ & 5 & 120\\
$C_5$ & 6 & 576\\
$C_3$ & 10 & 320\\
$C_2$ & 15 & 180\\
\hline
\end{tabular}
\end{center}
\label{classi_sgp}
\end{table}

The non-trivial subgroups of $\I$ are listed in Table \ref{ico_sgp}, together with their generators \cite{hoyle}.  Note that $\T$, $\D_{10}$ and $\D_6$ are maximal subgroups of $\I$, and that $\D_4$, $C_5$ and $C_3$ are normal subgroups of $\T$, $\D_{10}$ and $\D_6$, respectively (\cite{humpreys}, \cite{artin}). The permutation representations of the generators in $S_{12}$  are given in Table \ref{perm_rep} (see also Figure \ref{ico}).

Since $\I$ is a small group, its subgroup structure can be easily obtained in \texttt{GAP} by computing explicitly all its conjugacy classes of subgroups. 
In particular, there are $7$ classes of non-trivial subgroups in $\I$: any subgroup $H$ of $\I$ has the property that, if $K$ is another subgroup of $\I$ isomorphic to $H$, then $H$ and $K$ are conjugate in $\I$ (this property is referred to as
the "friendliness'' of the subgroup $H$, \cite{soicher}). In other words, denoting by $n_{\G}$ the number of subgroups of $\I$ isomorphic to $\G$, i.e. 
\begin{equation}\label{ng}
n_{\G} := | \{H < \I : H \simeq \G \}|,
\end{equation}  
we have (cf. Definition \ref{conjugacy_class})
\begin{equation*}
n_{\G} = |\C_{\I}(\G)|.
\end{equation*}

In Table \ref{classi_sgp} we list the order of each class of subgroups in $\I$. Geometrically,  different copies of $C_2$, $C_3$ and $C_5$ correspond to the different two-,three- and five-fold axes of the icosahedron, respectively.  In particular, different copies of $\D_{10}$ stabilise one of the 6 five-fold axes of the icosahedron, and each copy of $\D_6$ stabilises one of the 10 three-fold axes.  
Moreover, it is possible to inscribe 5 tetrahedra into a dodecahedron, and each different copy of the tetrahedral group in $\I$ stabilises
one of these tetrahedra.

\subsection{Subgroups of the crystallographic representations of $\I$}

Let $\G$ be a subgroup of $\I$. The function \eqref{iso} provides a representation of $\G$ in $B_6$, denoted by $\K_{\G}$, which is a subgroup of $\hat{\I}$. 
Let us denote by $\C_{B_6}(\K_{\G})$ the conjugacy class of $\K_{\G}$ in $B_6$. The next lemma shows that this class contains all the subgroups of the crystallographic representations of $\I$ in $B_6$. 

\begin{lemma}
Let $\h_i \in \C_{B_6}(\hat{\I})$ be a crystallographic representation of $\I$ in $B_6$, and let $\K_i \subseteq \h_i$ be a subgroup of $\h_i$ isomorphic to $\G$. Then $\K_i \in \C_{B_6}(\K_{\G})$.
\end{lemma} 
\begin{proof}
Since $\h_i \in C_{B_6}(\hat{\I})$, there exists $g \in B_6$ such that $g\h_ig^{-1} = \hat{\I}$, and therefore $g\K_ig^{-1} = \K'$ is a subgroup of $\hat{\I}$ isomorphic to $\G$. Since all these subgroups 
are conjugate in $\hat{\I}$ (they are "friendly" in the sense of \cite{soicher}), there exists $h \in \hat{\I}$ such that $h\K'h^{-1} = \K_{\G}$. Thus $(hg)\K_i(hg)^{-1} = \K_{\G}$, implying that $\K_i \in \C_{B_6}(\K_{\G})$.
\end{proof}

We next show that every element of $\C_{B_6}(\K_{\G})$ is a subgroup of a crystallographic representation of $\I$.
\begin{lemma}\label{lemma_sgp}
Let $\K_i \in \C_{B_6}(\K_{\G})$. There exists $\h_i \in \C_{B_6}(\hat{\I})$ such that $\K_i$ is a subgroup of $\h_i$.
\end{lemma}
\begin{proof}
Since $\K_i \in \C_{B_6}(\K_{\G})$, there exists $g \in B_6$ such that $g\K_ig^{-1} = \K_{\G}$. We define $\h_i :=g^{-1}\hat{\I} g$. It is immediate to see that $\K_i$ is a subgroup of $\h_i$.
\end{proof}

As a consequence of these Lemmata, $\C_{B_6}(\K_{\G})$ contains all the subgroups of $B_6$ which are isomorphic to $\G$ \emph{and} are subgroups of a crystallographic representation of $\I$. The explict forms of $\K_{\G}$ are given in the Appendix. We point out that it is possible to find subgroups of $B_6$ isomorphic to a subgroup $\G$ of $\I$ which are \emph{not} subgroups of any crystallographic representation of $\I$.  For example, the following subgroup
\begin{equation*}
\hat{\T} = \left< \left( \begin{array}{cccccc}
1 & 0 & 0 & 0 & 0 & 0 \\
0 & -1 & 0 & 0 & 0 & 0 \\
0 & 0 & -1 & 0 & 0 & 0 \\
0 & 0 & 0 & -1 & 0 & 0 \\
0 & 0 & 0 & 0 & -1 & 0 \\
0 & 0 & 0 & 0 & 0 & 1 
\end{array} \right), \left( \begin{array}{cccccc}
0 & 0 & -1 & 0 & 0 & 0 \\
1 & 0 & 0 & 0 & 0 & 0 \\
0 & -1 & 0 & 0 & 0 & 0 \\
0 & 0 & 0 & 0 & 0 & 1 \\
0 & 0 & 0 & 1 & 0 & 0 \\
0 & 0 & 0 & 0 & 1 & 0 
\end{array} \right) \right>.
\end{equation*}
is isomorphic to the tetrahedral group $\T$; a computation in \texttt{GAP} shows that it is not a subgroup of any elements in $\C_{B_6}(\hat{\I})$. Indeed, the two classes of subgroups, $\C_{B_6}(\K_{\T})$ and $\C_{B_6}(\hat{\T})$, are disjoint.

Using \texttt{GAP}, we compute the size of each $\C_{B_6}(\K_{\G})$ (see Table \ref{classi_sgp}). We observe that $|\C_{B_6}(\K_{\G})| < |\C_{B_6}(\hat{\I})| \cdot n_{\G}$.  This implies that crystallographic representations of $\I$ may share subgroups. 
In order to describe more precisely the subgroup structure of $\C_{B_6}(\hat{\I})$ we will use some basic results from graph theory and their spectra, which we are going to recall in the next section. 

\subsection{Some basic results of graph theory and their spectra}

In this section we recall, without proofs, some concepts and results from graph theory and spectral graph theory. Proofs and further results can be found, for example, in \cite{foulds} and \cite{doob}.

Let $G$ be a graph with vertex set $V= \{v_1, \ldots, v_n \}$. The number of edges incident with a vertex $v$ is called the \emph{degree} of $v$. If all vertices have the same degree $d$, then the graph is called \emph{regular of degree $d$}. 
A \emph{walk of length $l$} is a sequence of $l$ consecutive edges, and it is called a \emph{path} if they are all distinct. A \emph{circuit} is a path starting and ending at the same vertex, and the \emph{girth} of the graph is the length of the shortest circuit.
Two vertices $p$ and $q$ are \emph{connected} if there exists a path containing $p$ and $q$. The \emph{connected component} of a vertex $v$ is the set of all vertices connected to $v$.

The \emph{adjacency matrix} $A$ of $G$ is the $n \times n$ matrix $A=(a_{ij})$ whose entries $a_{ij}$ are equal to $1$ 
if the vertex $v_i$ is adjacent to the vertex $v_j$, and $0$ otherwise. It is immediate to see from its definition that $A$ is symmetric and $a_{ii}=0$ for all $i$, so that $\text{Tr}(A) = 0$. 
It follows that $A$ is diagonalisable and all its eigenvalues are real. The \emph{spectrum of the graph} is the set of all the eigenvalues of its adjacency matrix $A$, usually denoted by 
$\sigma(A)$. 

\begin{teo}\label{walks}
Let $A$ be the adjacency matrix of a graph $G$ with vertex set $V = \{v_1, \ldots, v_n \}$. Let $N_k(i,j)$ denote the number of walks of length $k$ starting at vertex $v_i$ and finishing at vertex $v_j$. We have
\begin{equation*}
N_k(i,j) = A^k_{ij}.
\end{equation*}
\end{teo}

We recall that the \emph{spectral radius} of a matrix $A$ is defined by $\rho(A):= \text{max} \{|\lambda| : \lambda \in \sigma(A) \}$. 
If $A$ is a non-negative matrix, i.e. if all its entries are non-negative, then $\rho(A) \in \sigma(A)$ \cite{johnson}. Since the adjacency matrix of a graph is non-negative,
$|\lambda| \leq \rho(A):=r$, where $\lambda \in \sigma(A)$ and $r$ is the largest eigenvalue. $r$ is called the \emph{index} of the graph $G$. 

\begin{teo}\label{regolare}
 Let $\lambda_1, \ldots, \lambda_n$ be the spectrum of a graph $G$, and let $r$ denote its index. Then $G$ is regular of degree $r$ if and only if 
 \begin{equation*}
  \frac{1}{n} \sum_{i=1}^n \lambda_i^2 = r.
 \end{equation*}
Moreover, if $G$ is regular the multiplicity of its index is equal to the number of its connected components.
\end{teo}

\subsection{Applications to the subgroup structure}

Let $\G$ be a subgroup of $\I$. In the following we represent the subgroup structure of the class of crystallographic representations of $\I$ in $B_6$, $\C_{B_6}(\hat{\I})$, as a graph. We say that $\h_1, \h_2 \in \C_{B_6}(\hat{\I})$ are 
adjacent to each other (i.e. connected by an edge) in the graph if there exists $P \in \C_{B_6}(\K_{\G})$ such that $P = \h_1 \cap \h_2$. We can therefore consider the graph $ G = (\C_{B_6}(\hat{\I}),E)$, where an edge $e \in E$ 
is of the form $(\h_1,\h_2)$. We call this graph \emph{$\G$-graph}.

Using \texttt{GAP}, we compute the adjacency matrices of the $\G$-graphs. The algorithms used are shown in the Appendix. The spectra of the $\G$-graphs are given in Table \ref{spectra_graphs}. 
We first of all notice that the adjacency matrix of the $C_5$-graph is the null matrix, implying that there are no two representations sharing precisely a subgroup isomorphic to $C_5$, i.e. not a subgroup containing $C_5$. We point out that, since the adjacency matrix of 
the $\D_{10}$-graph is not the null one, then there exist cystallographic representations, say $\h_i$ and $\h_j$, sharing a maximal subgroup isomorphic to $\D_{10}$. Since $C_5$ is a (normal) subgroup of $\D_{10}$, then $\h_i$ and $\h_j$ do share 
a $C_5$ subgroup, but also a $C_2$ subgroup. In other words, if two representations share a five-fold axis, then necessarily they also share a two-fold axis.    

\begin{table}[!t]
\caption{Spectra of the $\G$-graphs for $\G$ a non-trivial subgroup of $\I$ and $\G = \{e\}$, the trivial subgroup consisting of only the identity element $e$. The numbers highlighted are the indices of the graphs, and correspond to their degrees $d_{\G}$.}
\begin{center}
\begin{tabular}{|c|c|c|c|c|c|c|c|}
 \hline
 \multicolumn{2}{|c|}{$\T$-Graph} & \multicolumn{2}{|c|}{$\D_{10}$-Graph} & \multicolumn{2}{|c|}{$\D_6$-Graph} &  \multicolumn{2}{|c|}{$C_5$-graph} \\
 \hline
 Eig. & Mult. & Eig. & Mult. & Eig. & Mult. & Eig. & Mult. \\
 $\mathbf{5}$ & 1 & $\mathbf{6}$ & 6 & $\mathbf{10}$ & 6 & $\mathbf{0}$ & 192\\
 3 & 45 & 2 & 90 & 2 & 90 & & \\
 -3 & 45 & -2 & 90 & -2 & 90 & & \\
 1 & 50 & -6 & 6 & -10 & 6 & &\\ 
 - 1 & 50 & & & & & &\\
 -5 & 1 & & & & & &\\
\hline
 \multicolumn{2}{|c|}{$\D_4$-graph} & \multicolumn{2}{|c|}{$C_3$-graph} & \multicolumn{2}{|c|}{$C_2$-graph} & \multicolumn{2}{|c|}{$\{ e \}$-graph} \\
 \hline
 Eig. & Mult. & Eig & Mult. & Eig. & Mult. & Eig. & Mult. \\
  $\mathbf{30}$ & 1 & $\mathbf{20}$ & 2 & $\mathbf{60}$ & 2 & $\mathbf{60}$ & 1 \\
  18 & 5 & 4 & 90 & 4 & 90 & 12 & 5 \\
 12 & 5 & -4 & 100 & -4 & 90 & 4 & 90 \\
   6 & 15 & & & -12 & 10 & -4 & 90 \\
   2 & 45 & & & & & -12 & 5 \\
   0 & 31 & & & & & -60 & 1 \\
 -2 & 30 & & & & & & \\
 -4 & 45 & & & & & & \\
 -8 & 15 & & & & & & \\
\hline
\end{tabular}
\end{center}
\label{spectra_graphs}
\end{table}

A straightforward calculation based on Theorem \ref{regolare} leads to the following
\begin{prop}\label{reg}
 Let $\G$ be a subgroup of $\I$. Then the corresponding $\G$-graph is regular. 
\end{prop}

In particular, the degree $d_{\G}$ of each $\G$-graph is equal to the largest eigenvalue of the corresponding spectrum. As a consequence we have the following
\begin{prop}\label{sgp_rep}
Let $\h$ be a crystallographic representation of $\I$ in $B_6$. Then there are exactly $d_{\G}$ representations $\K_j \in \C_{B_6}(\hat{\I})$ such that 
\begin{equation*}
\h \cap \K_j = P_j, \quad \exists \; P_j \in \C(\K_{\G}), \quad j = 1, \ldots, d_{\G}.
\end{equation*}
In particular, we have $d_{\G} = 5,6,10,0,30,20,60$ and $60$ for $\G = \T,\D_{10},\D_6$, $C_5, \D_4, C_3, C_2$ and $\{e \}$, respectively.  
\end{prop}

In particular, this means that for any crystallographic representation of $\I$ there are precisely $d_{\G}$ other such representations which share a subgroup isomorphic to $\G$. In other words, 
we can associate to the class $\C_{B_6}(\hat{\I})$ the "subgroup matrix" $S$ whose entries are defined by
\begin{equation*}
S_{ij} = |\h_i \cap \h_j|, \qquad i,j =1,\ldots, 192.
\end{equation*}

The matrix $S$ is symmetric and $S_{ii} = 60$, for all $i$, since the order of $\I$ is 60. It follows from Proposition \ref{sgp_rep} that each row of $S$ contains $d_{\G}$ entries equal to $|\G|$. Moreover, a rearrangement of the columns of $S$  shows that the 192 crystallographic representations of $\I$ can be grouped into 12 sets of 16 such that any two of these representations in such a set of 16 share a $\D_4$-subgroup. This implies that the corresponding subgraph of the $\D_4$-graph is a \emph{complete graph}, i.e. every two distinct vertices are connected by an edge. 
From a geometric point of view, these 16 representations correspond to "6-dimensional icosahedra". This ensemble of 16 such icosahedra embedded into a six-dimensional hypercube can be viewed as 6D analogue of the 3D ensemble of five tetrahedra inscribed into a dodecahedron, sharing pairwise a $C_3$-subgroup.

We notice that, using Theorem \ref{regolare}, not all the graphs are connected. In particular, the $\D_{10}$- and the $\D_6$-graphs are made up of six connected components, whereas the $C_3$- and the $C_2$-graphs consist of two connected components.
With \texttt{GAP}, we implemented a  \emph{breadth-first search} (BFS) algorithm \cite{foulds}, which starts from a vertex $i$ and then ``scans'' for all the vertices connected to it, which allows us to find the connected components of a given $\G$-graph (see Appendix). 
We find that each connected component of the $\D_{10}$- and the $\D_6$-graphs is made up of $32$ vertices, while for the $C_3$ and $C_2$ each component consists of $96$ vertices.  For all other subgroups, the corresponding $\G$-graph is connected and the connected component contains trivially 192 vertices. 

We now consider in more detail the case when $\G$ is a maximal subgroup of $\I$. Let $\h \in \C_{B_6}(\hat{\I})$ and let us consider its vertex star in the corresponding $\G$-graph, i.e.
 \begin{equation}\label{vertex_star}
  V(\h) := \{ \K \in \C_{B_6}(\hat{\I}) : \K \; \text{is adjacent to} \; \h \}. 
 \end{equation}

A comparison of Tables \ref{ico_sgp} and \ref{spectra_graphs} shows that $d_{\G} = n_{\G}$ (i.e. the number of subgroups isomorphic to $\G$ in $\I$, cf. \eqref{ng}) and therefore, since the graph is regular, $|V(\h)| = d_{\G} = n_{\G}$. 
This suggests that there is a 1-1 correspondence between elements of the vertex star of $\h$ and subgroups of $\h$ isomorphic to $\G$; in other words, if we fix any subgroup $P$ of $\h$ isomorphic to $\G$, then $P$ "connects" $\h$ with exactly another representation $\K$.  
We thus have the following
\begin{prop}\label{reps}
 Let $\G$ be a maximal subgroup of $\I$. Then for every $P \in \C_{B_6}(\K_{\G})$ there exist \emph{exactly} two crystallographic representations of $\I$, $\h_1, \h_2 \in \C_{B_6}(\hat{\I})$,  such that $P = \h_1 
 \cap \h_2$.
\end{prop} 
In order to prove it, we first need the following lemma

\begin{lemma}\label{triangle} 
Let $\G$ be a maximal subgroup of $\I$. Then the corresponding $\G$-graph is triangle-free, i.e. it has no circuits of length three. 
\end{lemma}
\begin{proof}
Let $A_{\G}$ be the adjacency matrix of the $\G$-graph. By Theorem \ref{walks}, its third power $A_{\G}^3$ determines the number of walks of length 3, and in particular its diagonal entries, $(A^3_{\G})_{ii}$, for $i = 1, \ldots, 192$, correspond to the 
number of triangular circuits starting and ending in vertex $i$. A direct computation shows that $(A^3_{\G})_{ii} = 0$, for all $i$, thus implying the non-existence of triangular circuits in the graph.
\end{proof}

\begin{proof}[Proof of Proposition \ref{reps}] 
 If $P \in \C_{B_6}(\K_{\G})$, then, using Lemma \ref{lemma_sgp}, there exists $\h_1 \in \C_{B_6}(\hat{\I})$ such that $P$ is a subgroup of $\h_1$. Let us consider the vertex star $V(\h_1)$. We have $|V(\h_1)| = d_{\G}$; we call its elements $\h_2, \ldots, \h_{d_{\G}+1}$.
Let us suppose that $P$ is not a subgroup of any $\h_j$, for $j=2, \ldots, d_{\G}+1$. This implies that $P$ does not connect $\h_1$ with any of these $\h_j$. However, since $\h_1$ has exactly $n_{\G}$ different subgroups isomorphic to $\G$, then at least two vertices in the vertex star, say $\h_2$ and $\h_3$, are connected by the same subgroup isomorphic to $\G$, which we denote by $Q$. Therefore we have
\begin{equation*}
Q = \h_1 \cap \h_2, \quad Q = \h_1 \cap \h_3 \Rightarrow Q = \h_2 \cap \h_3.
\end{equation*}
This implies that $\h_1$, $\h_2$ and $\h_3$ form a triangular circuit in the graph, which is a contradiction due to Lemma \ref{triangle}, hence the result is proved.
\end{proof} 

It is noteworthy that the situation in $B_6^+$ is different. If we denote by $X_1$ and $X_2$ the two disjoint classes of crystallographic representations of $\I$ in $B_6^+$ (cf. \eqref{class_2}), we can build, in the
same way as described before, the $\G$-graphs for $X_1$ and $X_2$, for $\G = \T, \D_{10}$ and $\D_6$. The result is that the adjacency matrices of all these 6 graphs are the null matrix of dimension 96. This implies that these graphs 
have no edges, and so the representations in each class do not share any maximal subgroup of $\I$. As a consequence, we have the following:
\begin{prop}\label{class2}
Let $\h,\K \in \C_{B_6}(\hat{\I})$ be two crystallographic representations of $\I$, and $P = \h \cap \K$, $P \in \C_{B_6}(\K_{\G})$, where $\G$ is a maximal subgroup of $\I$. Then $\h$ and $\K$ are not conjugated in $B_6^+$.  In other words, the elements of $B_6$ which conjugate $\h$ with $\K$ are matrices with determinant equal to $-1$. 
\end{prop}

We conclude by showing a computational method which combines the result of Propositions \ref{class} and \ref{sgp_rep}. We first recall the following
\begin{defn}
Let $H$ be a subgroup of a group $G$. The \emph{normaliser} of $H$ in $G$ is given by
\begin{equation*}
N_G(H) := \{ g \in G : gHg^{-1} = H \}.
\end{equation*}
\end{defn}

\begin{cor}
Let $\h$ and $\K$ be two crystallographic representations of $\I$ in $B_6$, and $P \in \C(\K_{\G})$ such that $P = \h \cap \K$. Let 
\begin{equation*}
A_{\h,\K} = \{ M \in B_6 : M\h M^{-1} = \K \}
\end{equation*}
be the set of all the elements of $B_6$ which conjugate $\h$ with $\K$, and let $N_{B_6}(P)$ be the normaliser of $P$ in $B_6$. We have
\begin{equation*}
A_{\h,\K} \cap N_{B_6}(H) \neq \varnothing.
\end{equation*}
In other words, it is possible to find a non-trivial element $M \in B_6$  in the normaliser of $P$ in $B_6$ which conjugates $\h$ with $\K$.  
\end{cor}
\begin{proof}
Let us suppose $A_{\h,\K} \cap N_{B_6}(H) = \varnothing$. Then $M P M^{-1} \neq P$, for all $M \in A_{\h,\K}$. This implies, since $M \h M^{-1} = \K$, that $P$ is not a subgroup of $\K$, which is a contradiction.
\end{proof}

We give now an explicit example. We consider the representation $\hat{\I}$ as in \eqref{gen}, and its subgroup $\K_{\D_{10}}$ (the explicit form is given in the Appendix). With \texttt{GAP}, we find the other representation $\h_0 \in \C(\hat{\I})$ such that
$\K_{\D_{10}} = \hat{\I} \cap \h_0$. Its explicit form is given by
\begin{equation*}
\h_0 =  \left< \left( \begin{array}{cccccc}
0 & 0 & 0 & 0 & -1 & 0 \\
0 & 0 & 0 & 1 & 0 & 0 \\
0 & 0 & -1 & 0 & 0 & 0 \\
0 & 1 & 0 & 0 & 0 & 0 \\
-1 & 0 & 0 & 0 & 0 & 0 \\
0 & 0 & 0 & 0 & 0 & -1 
\end{array} \right), \left( \begin{array}{cccccc}
0 & 0 & 0 & 0 & -1 & 0 \\
0 & 0 & -1 & 0 & 0 & 0 \\
0 & 0 & 0 & 0 & 0 & 1\\
1 & 0 & 0 & 0 & 0 & 0 \\
0 & 0 & 0 & -1 & 0 & 0 \\
0 & -1 & 0 & 0 & 0 & 0 
\end{array} \right) \right>.
\end{equation*}   

A direct computation shows that the matrix 
\begin{equation*}
M =  \left( \begin{array}{cccccc}
1 & 0 & 0 & 0 & 0 & 0 \\
0 & -1 & 0 & 0 & 0 & 0 \\
0 & 0 & 1 & 0 & 0 & 0\\
0 & 0 & 0 & 1 & 0 & 0 \\
0 & 0 & 0 & 0 & 1 & 0 \\
0 & 0 & 0 & 0 & 0 & 1 
\end{array} \right) 
\end{equation*}
 belongs to $N_{B_6}(\K_{\D_{10}})$ and conjugate $\hat{\I}$ with $\h_0$. Note that $\text{det}M = -1$. 

\section{Conclusions}
In this work we explored the subgroup structure of the hyperoctahedral group. In particular we found the class of the crystallographic representations of the icosahedral group, whose order is 192. Any such representation, together with its corresponding projection operator $\pi^{\parallel}$, can be chosen to construct icosahedral quasicrystals via the cut-and-project method. We then studied in detail  the subgroup structure of this class. For this, 
we proposed a method based on spectral graph theory and introduced the concept of $\G$-graph, for a subgroup $\G$ of the icosahedral group. This allowed us to study the intersection and the subgroups shared by different representations. We have shown that, if we fix any representation $\h$ in the class and a maximal subgroup $P$ of $\h$, then there exists exactly one other representation $\K$ in the class such that $P = \h \cap \K$. As explained in the introduction, this can be used to describe transitions which keep intermediate symmetry encoded by $P$. In particular, this result implies in this context that a transition from a structure arising from $\h$ via projection will result in a structure obtainable for $\K$ via projection if the transition has intermediate symmetry described by $P$. 
Therefore, this setting is the starting point to analyse structural transitions between icosahedral quasicrystals, following the methods proposed in \cite{kramer}, \cite{katz} and \cite{paolo}, which we are planning to address in a forthcoming publication.

These mathematical tools have many applications also in other areas. A prominent example is virology. Viruses package their genomic material into protein containers with regular structures that can be modelled via lattices and group theory. Structural transitions of these containers, which involve rearrangements of the protein lattices, are important in rendering certain classes of viruses infective. As shown in \cite{giuliana}, such structural transitions can be modelled using projections of 6D icosahedral lattices and their symmetry properties. The results derived here therefore have a direct application to this scenario, and the information on the subgroup structure of the class of crystallographic representations of the icosahedral group and their intersections provides information on the symmetries of the capsid during the transition.

\section{Appendix A}
In order to render this paper self-contained, we provide the character tables of the subgroups of the icosahedral group, following \cite{artin} and \cite{fulton}.

\begin{itemize}
\item Tetrahedral group $\T$:

\begin{center}
\begin{tabular}{l|c c c c }
Irrep & $\mathcal C(e)$ & $4C_3$ & $4C_3^2$ & $3C_2$ \\
\hline
A & 1 & 1 & 1 & 1 \\
\multirow{2}*{E} 
& 1 & $\omega$ & $\omega^2$ & 1 \\
& 1 & $\omega^2$ & $\omega$ & 1 \\
T & 3 & 0 & 0 & -1
\end{tabular}
\end{center}

where $\omega = e^{\frac{2\pi i}{3}}$.

\item Dihedral group $\D_{10}$:

\begin{center}
\begin{tabular}{l|c c c c }
Irrep & $E$ & $2C_5$ & $2C_5^2$ & $5C_2$ \\
\hline
$A_1$ & 1 & 1 & 1 & 1 \\
$A_2$ & 1 & 1 & 1 & -1 \\
$E_1$ & 2 & $\gamma$ & $\gamma'$ & 0 \\
$E_2$ & 2 & $\gamma'$ & $\gamma$ & 0
\end{tabular}
\end{center}

with $\gamma = 2\text{cos}(\frac{2\pi}{5}) = \frac{\sqrt{5}-1}{2}$ and $\gamma' = 2\text{cos}(\frac{4\pi}{5}) = -\frac{\sqrt{5}+1}{2}$.

\item Dihedral group $\D_6$ (isomorphic to the symmetric group $S_3$):

\begin{center}
\begin{tabular}{l|c c c }
Irrep & $E$ & $3C_2$ & $2C_3$  \\
\hline
$A_1$ & 1 & 1 & 1 \\
$A_2$ & 1 & -1 & 1 \\
$E$ & 2 & 0 & -1
\end{tabular}
\end{center}

\item Cyclic group $C_5$:

\begin{center}
\begin{tabular}{l|c c c c c }
Irrep & e & $C_5$ & $C_5^2$ & $C_5^3$ & $C_5^4$   \\
\hline
A & 1 & 1 & 1 & 1 & 1  \\
\multirow{2}*{$E_1$} & 1 & $\epsilon$ & $\epsilon^2$ & $\epsilon^{2*}$ & $\epsilon^*$  \\
 & 1 & $\epsilon^*$ & $\epsilon^{2*}$ & $\epsilon^2$ & $\epsilon$  \\
\multirow{2}*{$E_2$} & 1 & $\epsilon^2$ & $\epsilon^*$ & $\epsilon$ & $\epsilon^{2*}$  \\
 & 1 & $\epsilon^{2*}$ & $\epsilon$ & $\epsilon^*$ & $\epsilon^2$  
\end{tabular}
\end{center}
where $\epsilon = e^{\frac{2\pi i}{5}}$. 

\item Dihedral group $D_4$ (the Klein Four Group):

\begin{center}
\begin{tabular}{l|c c c c}
Irrep & E& $C_{2x}$ & $C_{2y}$ & $C_{2z}$    \\
\hline
$A$ & 1 & 1 & 1 & 1 \\
$B_1$ & 1 & 1 & -1 & -1 \\
$B_2$ & 1 & -1 & 1 & -1 \\
$B_3$ & 1 & -1 & -1 & 1 
\end{tabular}
\end{center}

\item Cyclic group $C_3$:

\begin{center}
\begin{tabular}{l|c c c}
Irrep & E& $C_3$ & $C_3^2$    \\
\hline
$A$ & 1 & 1 & 1 \\
 \multirow{2}*{E} & 1 & $\eta$ & $\eta^2$ \\
 & 1 & $\eta^2$ & $\eta$  
\end{tabular}
\end{center}
where $\eta = e^{\frac{2\pi i}{3}}$.

\item Cyclic group $C_2$: 
\begin{center}
\begin{tabular}{l|c c}
Irrep & E & $C_2$ \\
\hline
$A$ & 1 & 1 \\
$B$ & 1 & -1 
\end{tabular}
\end{center}

\end{itemize}

\section{Appendix B}
In this Appendix we show the explicit forms of $\K_{\G}$, the representations in $B_6$ of the subgroups of $\I$, together with their decompositions in $GL(6,\R)$. 
\begin{equation*}
\K_{\T} = \left< \left( \begin{array}{cccccc}
0 & 0 & 0 & 0 & 0 & 1 \\
0 & 0 & 0 & 0 & 1 & 0 \\
0 & 0 & -1 & 0 & 0 & 0 \\
0 & 0 & 0 & -1 & 0 & 0 \\
0 & 1 & 0 & 0 & 0 & 0 \\
1 & 0 & 0 & 0 & 0 & 0 
\end{array} \right), \left( \begin{array}{cccccc}
0 & 1 & 0 & 0 & 0 & 0 \\
0 & 0 & 0 & -1 & 0 & 0 \\
0 & 0 & 0 & 0 & 0 & -1\\
-1 & 0 & 0 & 0 & 0 & 0 \\
0 & 0 & 1 & 0 & 0 & 0 \\
0 & 0 & 0 & 0 & -1 & 0 
\end{array} \right) \right>,
\end{equation*}
\begin{equation*}
\K_{\D_{10}} = \left< \left( \begin{array}{cccccc}
0 & 0 & 0 & 0 & 0 & -1 \\
0 & -1 & 0 & 0 & 0 & 0 \\
0 & 0 & 0 & 1 & 0 & 0 \\
0 & 0 & 1 & 0 & 0 & 0 \\
0 & 0 & 0 & 0 & -1 & 0 \\
-1 & 0 & 0 & 0 & 0 & 0 
\end{array} \right), \left( \begin{array}{cccccc}
0 & 0 & 0 & 0 & 0 & 1 \\
0 & 1 & 0 & 0 & 0 & 0 \\
0 & 0 & 0 & 0 & -1 & 0\\
-1 & 0 & 0 & 0 & 0 & 0 \\
0 & 0 & 0 & 1 & 0 & 0 \\
0 & 0 & 1 & 0 & 0 & 0 
\end{array} \right) \right>,
\end{equation*}
\begin{equation*}
\K_{\D_{6}} = \left< \left( \begin{array}{cccccc}
0 & 0 & 0 & 0 & 0 & -1 \\
0 & -1 & 0 & 0 & 0 & 0 \\
0 & 0 & 0 & 1 & 0 & 0 \\
0 & 0 & 1 & 0 & 0 & 0 \\
0 & 0 & 0 & 0 & -1 & 0 \\
-1 & 0 & 0 & 0 & 0 & 0 
\end{array} \right),  \left( \begin{array}{cccccc}
0 & 0 & 0 & 0 & 0 & 1 \\
0 & 0 & 0 & 1 & 0 & 0 \\
0 & -1 & 0 & 0 & 0 & 0\\
0 & 0 & -1 & 0 & 0 & 0 \\
1 & 0 & 0 & 0 & 0 & 0 \\
0 & 0 & 0 & 0 & 1 & 0 
\end{array} \right) \right>,
\end{equation*}
\begin{equation*}
\K_{C_5} = \left< \left( \begin{array}{cccccc}
0 & 0 & 0 & 0 & 0 & 1 \\
0 & 1 & 0 & 0 & 0 & 0 \\
0 & 0 & 0 & 0 & -1 & 0\\
-1 & 0 & 0 & 0 & 0 & 0 \\
0 & 0 & 0 & 1 & 0 & 0 \\
0 & 0 & 1 & 0 & 0 & 0 
\end{array} \right) \right>,
\end{equation*}
\begin{equation*}
\K_{\D_{4}} = \left< \left( \begin{array}{cccccc}
0 & 0 & 0 & 0 & 0 & -1 \\
0 & -1 & 0 & 0 & 0 & 0 \\
0 & 0 & 0 & 1 & 0 & 0 \\
0 & 0 & 1 & 0 & 0 & 0 \\
0 & 0 & 0 & 0 & -1 & 0 \\
-1 & 0 & 0 & 0 & 0 & 0 
\end{array} \right),  \left( \begin{array}{cccccc}
0 & 0 & 0 & 0 & 0 & 1 \\
0 & 0 & 0 & 0 & 1 & 0 \\
0 & 0 & -1 & 0 & 0 & 0 \\
0 & 0 & 0 & -1 & 0 & 0 \\
0 & 1 & 0 & 0 & 0 & 0 \\
1 & 0 & 0 & 0 & 0 & 0 
\end{array} \right) \right>.
\end{equation*}
\begin{equation*}
\K_{C_3} = \left< \left( \begin{array}{cccccc}
0 & 0 & 0 & 0 & 0 & 1 \\
0 & 0 & 0 & 1 & 0 & 0 \\
0 & -1 & 0 & 0 & 0 & 0\\
0 & 0 & -1 & 0 & 0 & 0 \\
1 & 0 & 0 & 0 & 0 & 0 \\
0 & 0 & 0 & 0 & 1 & 0 
\end{array} \right) \right>,
\end{equation*}
\begin{equation*}
\K_{C_2} = \left< \left( \begin{array}{cccccc}
0 & 0 & 0 & 0 & 0 & 1 \\
0 & 0 & 0 & 0 & 1 & 0 \\
0 & 0 & -1 & 0 & 0 & 0 \\
0 & 0 & 0 & -1 & 0 & 0 \\
0 & 1 & 0 & 0 & 0 & 0 \\
1 & 0 & 0 & 0 & 0 & 0 
\end{array} \right) \right>.
\end{equation*}

\begin{equation*}
\K_{\T} \simeq 2T, \qquad  \K_{\D_{10}} \simeq 2A_2 \oplus E_1 \oplus E_2, \qquad  \K_{\D_6} \simeq 2A_2 \oplus 2E, \qquad  \K_{\C_5} \simeq 2A \oplus E_1 \oplus E_2,
\end{equation*}
\begin{equation*}
\K_{\D_4} \simeq 2B_1 \oplus 2B_2 \oplus 2B_3, \qquad \K_{\C_3} \simeq 2A \oplus 2E, \qquad \K_{\C_2} \simeq 2A \oplus 4B.
\end{equation*}

\section{Appendix C}
In this Appendix we show our algorithms, which have been implemented in \texttt{GAP} and used in various sections of the paper. We list them with a number from 1 to 5. 
\begin{enumerate}
\item Classification of the crystallographic representations of $\I$ (see Section \ref{cube_section}). The algorithm carries out steps 1-4 used to prove Proposition \ref{class}.
\item Computation of the vertex star of a given vertex $i$ in the $\G$-graphs. In the following, $H$ stands for the class $\C_{B_6}(\hat{\I})$ of the crystallographic representations of $\I$, $i \in \{1, \ldots, 192 \}$ denotes a vertex in the $\G$-graph corresponding to the representation $H[i]$ and $n$ stands for the
size of $\G$: we can use the size instead of the explicit form of the subgroup since, in the case of the icosahedral group, all the non isomorphic subgroups have different sizes.
\item Computation of the adjacency matrix of the $\G$-graph.
\item BFS algorithm for the computation of the connected component of a given vertex $i$ of the $\G$-graph.  
\item Computation of \emph{all} connected components of a $\G$-graph.
\end{enumerate}

\noindent {\bf{Algorithm 1}}: \\

{\small{
\noindent \texttt{gap >   B6:= Group([(1,2)(7,8),(1,2,3,4,5,6)(7,8,9,10,11,12),(6,12)]); \\
gap >   C:= ConjugacyClassesSubgroups(B6); \\
gap >  C60:= Filtered(C,x->Size(Representative(x))=60); \\
gap >  Size(C60); \\
 3 \\
gap >  s60:= List(C60,Representative); \\
gap > I:= AlternatingGroup(5); \\
gap> IsomorphismGroups(I,s60[1]); \\
 $[ (2,4)(3,5), (1,2,3) ] -> [ (1,3)(2,4)(7,9)(8,10), (3,10,11)(4,5,9) ]$ \\ 
gap>   IsomorphismGroups(I,s60[2]);\\
 $[ (2,4)(3,5), (1,2,3) ] -> [ (1,2)(3,10)(4,9)(5,11)(6,12)(7,8), (1,2,4)(3,12,5)(6,11,9)(7,8,10) ]$ \\
gap >  IsomorphismGroups(I,s60[3]); \\
 $[ (2,4)(3,5), (1,2,3) ] -> [ (2,6)(4,11)(5,10)(8,12), (1,3,5)(2,4,6)(7,9,11)(8,10,12) ]$ \\
gap> CB6s60:= ConjugacyClassSubgroups(B6,s60[2]); \\
gap>  Size(CB6s60); \\
 192}}}\\

\noindent {\bf{Algorithm 2}}: \\

{\small{
\noindent \texttt{gap> vertex star:=function(H,i,n) \\
> local j,R,S; \\
> R:=[]; \\
> for j in [1..Size(H)] do \\
> S:=Intersection(H[i],H[j]); \\
> if Size(S) = n then \\
> R:=Concatenation(R,[j]); \\
> fi; \\
> od; \\
> return R; \\
> end;}}} \\

\noindent {\bf{Algorithm 3}}: \\

{\small{
\noindent \texttt{ gap> adjacency matrix:=function(H,n) \\
> local i,j,C,A; \\
> A:=NullMat(Size(H),Size(H)); \\
> for i in [1..Size(H)] do \\
> C:=vertex star(H,i,n); \\
> for j in [1..Size(C)] do \\
> A[i][C[j]]:=1; \\
> od; \\
> od; \\
> return A; \\
> end;}}}\\

\noindent {\bf{Algorithm 4}}: \\

{\small{
\noindent \texttt{gap> connected component:=function(H,i,n) \\
> local R,S,T,j,k,C; \\
> R:=[i]; \\
> S:=[i]; \\
> while Size(S) <= Size(H) do \\
> T:=[]; \\
> for j in [1..Size(R)] do \\
> C:=vertex star(H,R[j],n); \\
> for k in [1..Size(C)] do \\
> if (C[k] in S) = false then \\
> Add(S,C[k]); \\
> T:=Concatenation(T,[C[k]]); \\
> fi; \\
> od; \\
> od; \\
> if T = [] then return S; \\
> else \\
> R:=T; \\
> fi; \\
> od; \\
> return S; \\
> end;}}} \\

\noindent {\bf{Algorithm 5}}: \\

{\small{
\noindent \texttt{gap> connected components:=function(H,n)\\
> local j,S,C;\\
> C:=[connected component(H,1,n)];\\
> S:=Flat(C);\\
> if Size(S) = Size(H) then return S;\\
> fi;\\
> for j in [1..Size(H)] do\\
> if (j in S) = false then\\
> C:=Concatenation(C,[connected component(H,j,n)]);\\
> S:=Flat(C);\\
> if Size(S) = Size(H) then return C;\\
> fi;\\
> fi;\\
> od;\\
> end;}}}\\

{\bf{Acknowledgements}}. We would like to thank Silvia Steila, Pierre-Philippe Dechant, Paolo Cermelli and Giuliana Indelicato for useful discussions, and Paolo Barbero  and Alessandro Marchino for technical help.

\bibliography{zappa_bibl}{}
\bibliographystyle{plain}

\end{document}